\documentclass[12pt,a4paper]{amsart} 
\usepackage{inputenc}
\usepackage{latexsym}
\usepackage{pinlabel}          
\usepackage[a4paper , lmargin = {2.5cm} , rmargin = {2.5cm} , tmargin = {2.5cm} , bmargin = {2.5cm} ]{geometry}             
\usepackage{graphicx}
\usepackage{amssymb, amsthm}
\usepackage{epstopdf}
\usepackage{romannum}
\usepackage{tikz}
\usepackage{subfig}
\usepackage[arrow, matrix, curve]{xy} 
\usepackage{hyperref}
\usetikzlibrary{arrows}
\usepackage{enumerate}  
 \usepackage{graphicx}
 \usepackage{caption}
 \usepackage{bm}

\DeclareMathOperator{\Aut}{Aut}

\theoremstyle{plain}

\newtheorem{theorem}{Theorem}[section]
\newtheorem{lemma}[theorem]{Lemma}

\newtheorem{corollary}[theorem]{Corollary}

\newtheorem{proposition}[theorem]{Proposition}
\theoremstyle{definition}
\newtheorem{question}[theorem]{Question}
\newtheorem{remark}[theorem]{Remark}
\newtheorem{definition}[theorem]{Definition}

\definecolor{applegreen}{rgb}{0.55, 0.71, 0.0}
\definecolor{dgreen}{rgb}{0.0, 0.5, 0.0}

\newcommand{\Z}{\ensuremath{\mathbb{Z}}}

\title[]{On quotients of Coxeter groups}

\author{Philip Möller and Olga Varghese}
\date{\today}

\address{Philip M\"oller\\
Department of Mathematics\\
University of M\"unster\\ 
Einsteinstra\ss e 62\\
48149 M\"unster (Germany)}
\email{philip.moeller@uni-muenster.de}

\address{Olga Varghese\\ Institute of Mathematics, Heinrich-Heine-University Düsseldorf, Universitätsstra{\upshape{\ss}}e 1, 40225, Düsseldorf (Germany)}
\email{olga.varghese@hhu.de}

\begin{document}
	
\pagenumbering{arabic}
	
	\begin{abstract}
	A group $G$ is said to be \emph{just infinite} if $G$ itself is infinite but all proper quotients of $G$ are finite. We show that a Coxeter group $W_\Gamma$ is just infinite if and only if $\Gamma$ is isomorphic to one of the following graphs: $\widetilde{A}_1$, $\widetilde{A}_n (n\geq 2)$, $\widetilde{B}_n (n\geq 3)$, $\widetilde{C}_n (n\geq 2)$, $\widetilde{D}_n(n\geq 4)$, $\widetilde{E}_6$, $\widetilde{E}_7$, $\widetilde{E}_8$, $\widetilde{F}_4$ or $\widetilde{G}_2$. 
    
    Moreover, we show that just infinite Coxeter groups are profinitely rigid among all Coxeter groups.

\vspace{1cm}
\hspace{-0.6cm}
\medskip

\medskip
\hspace{-0.5cm}{\bf 2010 Mathematics Subject Classification.} Primary: 20F55; Secondary: 20E26.

\thanks{The first author is funded
	by a stipend of the Studienstiftung des deutschen Volkes and  by the Deutsche Forschungsgemeinschaft (DFG, German Research Foundation) under Germany's Excellence Strategy EXC 2044--390685587, Mathematics M\"unster: Dynamics-Geometry-Structure. The second author is supported by DFG grant VA 1397/2-2. This work is part of the PhD project of the first author.}

\end{abstract}

\maketitle

\section{Introduction}
Following Gilbert Baumslag  a group is called \emph{just infinite} if it is infinite and all of its proper quotients are finite. The infinite cyclic group $\mathbb{Z}$ and the infinite dihedral group $D_\infty=\mathbb{Z}/2\mathbb{Z}*\mathbb{Z}/2\mathbb{Z}$ are examples of just infinite groups. Other examples are the groups ${\rm PSL}_n(\mathbb{Z})$ for $n\geq 3$ \cite{Margulis}, the Grigorchuk group \cite{Grigorchuk} and the Nottingham group \cite{Hegedus}, \cite{Klopsch}. This question was also studied in \cite{McCarthy} and solved for group extensions of abelian groups.

Here we are interested in a combinatorial characterization of just infinite Coxeter groups. Coxeter groups are defined via edge-labeled graphs.

Given a finite graph $\Gamma$ with the vertex set $V(\Gamma)$, the edge set $E(\Gamma)$ and an edge-labeling $m\colon E(\Gamma)\to\mathbb{N}_{\geq 3}\cup\left\{\infty\right\}$, the associated \emph{Coxeter group} $W_\Gamma$ is given be the presentation
\begin{gather*}
W_\Gamma=\langle V(\Gamma)\mid v^2\text{ for all }v\in V(\Gamma), (vw)^2 \text{ if }\left\{v,w\right\}\notin E(\Gamma), (vw)^{m(\left\{v,w\right\})}\text{ if } \left\{v,w\right\}\in E(\Gamma)\\ \text{ and } m(\left\{v,w\right\}<\infty)\rangle\,
\end{gather*}
Since the edge-label $3$ appears very often in the study of Coxeter groups it is convenient to omit this label in the graph $\Gamma$. Two special cases to keep in mind are that $W_\Gamma$ is isomorphic to a free product of copies of $\mathbb{Z}/2\mathbb{Z}$ if any two vertices of $\Gamma$ are linked by an edge and the edge label of every edge is $\infty$, and $W_\Gamma$ is isomorphic to a direct product of copies of $\mathbb{Z}/2\mathbb{Z}$ if $\Gamma$ is discrete.

It is natural to try to relate group theoretical properties of $W_\Gamma$ to combinatorial properties of $\Gamma$. Examples of this are Serre's property FA \cite[§6]{Serre}, being finite \cite{Coxeter}, being virtually abelian \cite[§4]{Humphreys} or being Gromov-hyperbolic \cite{Moussong}. More importantly, being virtually indicable, that is containing a finite index subgroup which in turn has $\Z$ as a quotient is closely related to the structure of $\Gamma$, see \cite{Cooper}. This property will be used to characterize the just infinite Coxeter groups.

We show that
whether $W_\Gamma$ is just infinite or not can also be easily read off $\Gamma$.

\newpage
\begin{theorem}
\label{JustInfiniteCoxeter}
A Coxeter group $W_\Gamma$ is just infinite if and only if $\Gamma$ is isomorphic to one of the graphs in Figure 1. 

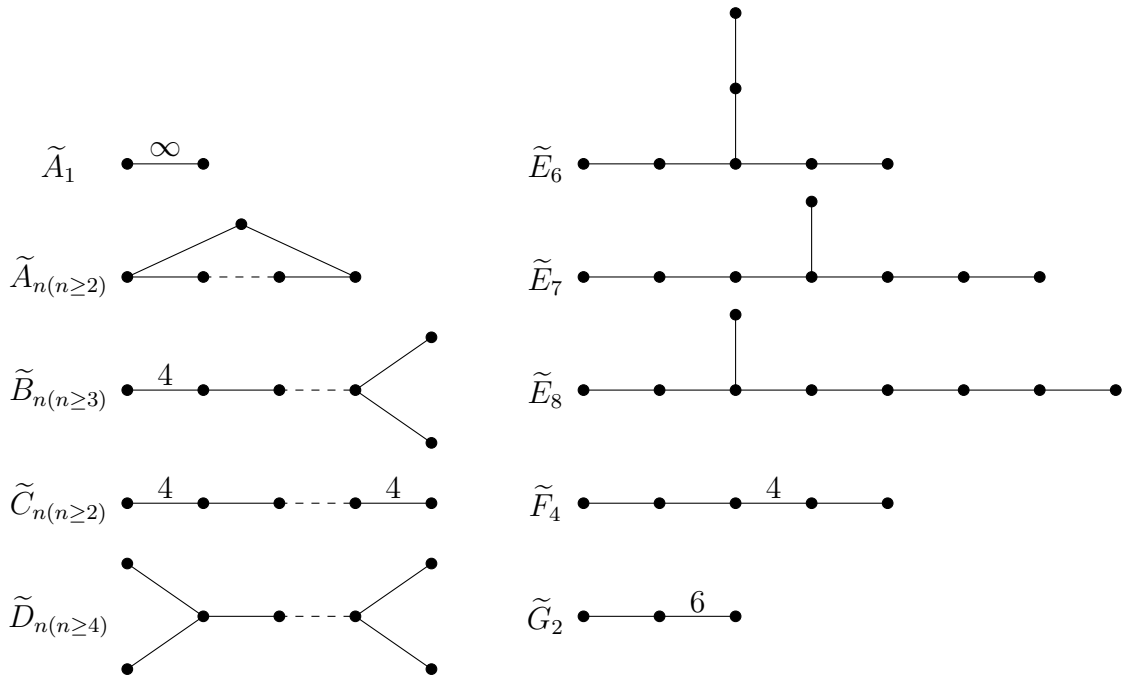
\begin{figure}[h]
	\begin{center}
	\captionsetup{justification=centering}
		\begin{tikzpicture}[scale=1, transform shape]
			\draw[fill=black]  (0,0) circle (2pt);
			\draw[fill=black]  (1,0) circle (2pt);
			\draw (0,0)--(1,0);
			\node at (0.5,0.2){$\infty$};
			\node at (-0.9, 0) {$\widetilde{A}_1$};	

            \draw[fill=black] (0,-1.5) circle (2pt);
            \draw[fill=black] (1,-1.5) circle (2pt);
            \draw (0,-1.5)--(1,-1.5);
            \draw[fill=black] (2,-1.5) circle (2pt);
            \draw[fill=black] (3,-1.5) circle (2pt);
            \draw[dashed] (1,-1.5)--(2, -1.5);
            \draw (2,-1.5)--(3,-1.5);
            \draw[fill=black] (1.5, -0.8) circle (2pt);
            \draw (0,-1.5)--(1.5, -0.8);
            \draw (3,-1.5)--(1.5, -0.8);
            \node at (-0.9, -1.5) {$\widetilde{A}_{n (n\geq 2)}$};	

            \draw[fill=black] (0,-3) circle (2pt);
            \draw[fill=black] (1,-3) circle (2pt);
            \node at (0.5,-2.8) {$4$};
            \draw (0,-3)--(1,-3);
            \draw (1,-3)--(2,-3);
            \draw[dashed] (2,-3)--(3,-3);
            \draw[fill=black] (3,-3) circle (2pt);
            \draw[fill=black] (2,-3) circle (2pt);
            \draw[fill=black] (4, -2.3) circle (2pt);
            \draw[fill=black] (4, -3.7) circle (2pt);
            \draw (3, -3)--(4,-2.3);
            \draw (3,-3)--(4, -3.7);
            \node at (-0.9, -3) {$\widetilde{B}_{n (n\geq 3)}$};	

            \draw[fill=black] (0,-4.5) circle (2pt);
            \draw[fill=black] (1,-4.5) circle (2pt);
            \draw (0,-4.5)--(1, -4.5);
            \node at (0.5, -4.3) {$4$};
            \draw[fill=black] (2, -4.5) circle (2pt);
            \draw (1, -4.5)--(2,-4.5);
            \draw[dashed] (2,-4.5)--(3,-4.5);
            \draw[fill=black] (3,-4.5) circle (2pt);
            \draw[fill=black] (4,-4.5) circle (2pt);
            \draw (3,-4.5)--(4,-4.5);
            \node at (3.5, -4.3) {$4$};
            \node at (-0.9, -4.5) {$\widetilde{C}_{n (n\geq 2)}$};

            \draw[fill=black] (0,-5.3) circle (2pt);
            \draw[fill=black] (0,-6.7) circle (2pt);
            \draw[fill=black] (1,-6) circle (2pt);
            \draw (0,-5.3)--(1,-6);
            \draw (0,-6.7)--(1,-6);
            \draw[fill=black] (2,-6) circle (2pt);
            \draw (1,-6)--(2,-6);
            \draw[dashed] (2,-6)--(3,-6);
            \draw[fill=black] (3,-6) circle (2pt);
            \draw[fill=black] (4, -5.3) circle (2pt);
            \draw[fill=black] (4, -6.7) circle (2pt);
            \draw (3,-6)--(4, -5.3);
            \draw (3,-6)--(4, -6.7);           
            \node at (-0.9, -6) {$\widetilde{D}_{n (n\geq 4)}$};

            \draw[fill=black] (6,0) circle (2pt);
            \draw[fill=black] (7,0) circle (2pt);
            \draw[fill=black] (8,0) circle (2pt);
            \draw[fill=black] (9,0) circle (2pt);
            \draw[fill=black] (10,0) circle (2pt);
            \draw[fill=black] (8,1) circle (2pt);
            \draw[fill=black] (8,2) circle (2pt);
            \draw (6,0)--(10,0);
            \draw (8,0)--(8,2);
            \node at (5.5,0) {$\widetilde{E}_6$};

            \draw[fill=black] (6, -1.5) circle (2pt);
            \draw[fill=black] (7, -1.5) circle (2pt);
            \draw[fill=black] (8, -1.5) circle (2pt);
            \draw[fill=black] (9, -1.5) circle (2pt);
            \draw[fill=black] (10, -1.5) circle (2pt);
            \draw[fill=black] (11, -1.5) circle (2pt);
            \draw[fill=black] (12, -1.5) circle (2pt);
            \draw[fill=black] (9, -0.5) circle (2pt);
            \draw (6,-1.5)--(12, -1.5);
            \draw (9, -1.5)--(9, -0.5);
            \node at (5.5,-1.5) {$\widetilde{E}_7$};

            \draw[fill=black] (6, -3) circle (2pt);
            \draw[fill=black] (7, -3) circle (2pt);
            \draw[fill=black] (8, -3) circle (2pt);
            \draw[fill=black] (9, -3) circle (2pt);
            \draw[fill=black] (10, -3) circle (2pt);
            \draw[fill=black] (11, -3) circle (2pt);
            \draw[fill=black] (12, -3) circle (2pt);
            \draw[fill=black] (13, -3) circle (2pt);
            \draw[fill=black] (8, -2) circle (2pt);
            \draw (6,-3)--(13, -3);
            \draw (8, -3)--(8, -2);
            \node at (5.5,-3) {$\widetilde{E}_8$};

            \draw[fill=black] (6, -4.5) circle (2pt);
            \draw[fill=black] (7, -4.5) circle (2pt);
            \draw[fill=black] (8, -4.5) circle (2pt);
            \draw[fill=black] (9, -4.5) circle (2pt);
            \draw[fill=black] (10, -4.5) circle (2pt);
            \draw (6, -4.5)--(10, -4.5);
            \node at (8.5, -4.3) {$4$};
            \node at (5.5,-4.5) {$\widetilde{F}_4$};

            \draw[fill=black] (6, -6) circle (2pt);
            \draw[fill=black] (7, -6) circle (2pt);
            \draw[fill=black] (8, -6) circle (2pt);
            \draw (6,-6)--(8, -6);
            \node at (7.5, -5.8) {$6$};
            \node at (5.5,-6) {$\widetilde{G}_2$};   
		\end{tikzpicture}
	\caption{Coxeter graphs of type $\widetilde{X}_n$ where $|V(\widetilde{X}_n)|=n+1$.}
	\end{center}
\end{figure}
\end{theorem}

The associated Coxeter groups to the graphs in Figure 1 are precisely the \emph{irreducible affine Coxeter groups}. We show that these groups can be distinguished among all Coxeter groups by their finite quotients. 
For a group $G$  we denote by $\mathcal{F}(G)$ the set of isomorphism classes of finite quotients of $G$. Two groups $G$ and $H$ are said to have the same finite quotients if $\mathcal{F}(G) =\mathcal{F}(H)$. A group $G$ is called \emph{profinitely rigid among a class of groups $\mathcal{C}$} if $G\in \mathcal{C}$ and for any group $H$ in the class $\mathcal{C}$ whenever $\mathcal{F}(G)=\mathcal{F}(H)$, then $G\cong H$. The profinite rigidity among the class consisting of finitely generated residually finite groups of some Coxeter groups where $|V(\Gamma)|=3$  was shown in \cite[Corollary 4.5]{Bridson}. It was proven in \cite[Theorem 1.2]{SantosRegoSchwer} that a Coxeter group $W_\Gamma$ where $|V(\Gamma)|\leq 3$ is profinitely rigid among  the class consisting of Coxeter groups where the defining graphs have at most $3$ vertices. We extend these results to just infinite Coxeter groups.

\begin{theorem}
Irreducible affine Coxeter groups are profinitely rigid among the class consisting of all Coxeter groups.
\end{theorem}

By definition, a finitely generated residually finite group $G$ is called \emph{profinitely rigid (in the absolute sense)} if $G$ is profinitely rigid among all finitely generated residually finite groups. We end the introduction with the following question.
\begin{question}
Are Coxeter groups profinitely rigid in the absolute sense?
\end{question}

\section{Coxeter groups} 
Coxeter groups have been studied from many different perspectives, e.g. in \cite{Coxeter}, \cite{Tits}, \cite{Davis}. More information about Coxeter groups can be found in \cite{Humphreys}. 

A Coxeter group $W_\Gamma$ is called \emph{irreducible} if $\Gamma$ is connected. Note that if $\Gamma$ has the connected components $\Gamma_1,\ldots, \Gamma_n$, then $W_\Gamma\cong W_{\Gamma_1}\times\ldots\times W_{\Gamma_n}$. Hence, the irreducible Coxeter groups can be seen as the building blocks of a general Coxeter group. For us, it will be important that there exist three different types of Coxeter groups: spherical, affine and general Coxeter groups. The spherical ones are precisely the finite ones and the affine ones are precisely the Coxeter groups, which are virtually free abelian of positive rank.
These Coxeter groups have been classified using their  
Coxeter graphs.

\begin{theorem}(\cite{Coxeter}, \cite[Proposition 17.2.1]{Davis}, \cite[§4]{Humphreys})
\label{FiniteVirtuallyAbelian}
Let $W_\Gamma$ be a Coxeter group. 
\begin{enumerate}
\item The group $W_\Gamma$ is finite if and only if every connected component of $\Gamma$ is isomorphic to one of the graphs in Figure 2.
\begin{figure}[h]
	\begin{center}
	\captionsetup{justification=centering}
		\begin{tikzpicture}
			\draw[fill=black]  (0,0) circle (2pt);
			\draw[fill=black]  (1,0) circle (2pt);
            \draw[fill=black]  (2,0) circle (2pt);
			\draw (0,0)--(1,0);
            \draw[dashed] (1,0)--(2,0);
            \draw[fill=black] (3,0) circle (2pt);
            \draw (2,0)--(3,0);
			\node at (-0.9,0) {$A_{n(n\geq 1)}$};

            \draw[fill=black]  (0,-1.5) circle (2pt);
			\draw[fill=black]  (1,-1.5) circle (2pt);
            \draw[fill=black]  (2,-1.5) circle (2pt);
			\draw (0,-1.5)--(1,-1.5);
            \draw[dashed] (1,-1.5)--(2,-1.5);
            \draw[fill=black] (3,-1.5) circle (2pt);
            \draw (2,-1.5)--(3,-1.5);
            \node at (2.5, -1.3) {$4$};
            \node at (-0.9,-1.5) {$B_{n(n\geq 2)}$};

            \draw[fill=black]  (0,-3) circle (2pt);
			\draw[fill=black]  (1,-3) circle (2pt);
            \draw[fill=black]  (2,-3) circle (2pt);
			\draw (0,-3)--(1,-3);
            \draw[dashed] (1,-3)--(2,-3);
            \draw[fill=black] (3,-2.3) circle (2pt);
            \draw[fill=black] (3,-3.7) circle (2pt);
            \draw (2,-3)--(3, -2.3);
            \draw (2,-3)--(3, -3.7);
            \node at (-0.9,-3) {$D_{n(n\geq 4)}$};

            \draw (0,-4.5)--(4,-4.5);
            \draw (2,-3.5)--(2,-4.5);
            \draw[fill=black] (0,-4.5) circle (2pt); 
            \draw[fill=black] (1,-4.5) circle (2pt);
            \draw[fill=black] (2,-4.5) circle (2pt);
            \draw[fill=black] (3,-4.5) circle (2pt);
            \draw[fill=black] (4,-4.5) circle (2pt);
            \draw[fill=black] (2,-3.5) circle (2pt);
            \node at (-0.5,-4.5) {$E_6$};

            \draw (0,-6)--(5,-6);
            \draw (2,-5)--(2,-6);
            \draw[fill=black] (0,-6) circle (2pt); 
            \draw[fill=black] (1,-6) circle (2pt);
            \draw[fill=black] (2,-6) circle (2pt);
            \draw[fill=black] (3,-6) circle (2pt);
            \draw[fill=black] (4,-6) circle (2pt);
            \draw[fill=black] (5,-6) circle (2pt);
            \draw[fill=black] (2,-5) circle (2pt);
            \node at (-0.5,-6) {$E_7$};

            \draw (0,-7.5)--(6,-7.5);
            \draw (2,-6.5)--(2,-7.5);
            \draw[fill=black] (0,-7.5) circle (2pt); 
            \draw[fill=black] (1,-7.5) circle (2pt);
            \draw[fill=black] (2,-7.5) circle (2pt);
            \draw[fill=black] (3,-7.5) circle (2pt);
            \draw[fill=black] (4,-7.5) circle (2pt);
            \draw[fill=black] (5,-7.5) circle (2pt);
            \draw[fill=black] (6,-7.5) circle (2pt);
            \draw[fill=black] (2,-6.5) circle (2pt);
            \node at (-0.5,-7.5) {$E_8$};

            \draw[fill=black] (8,0) circle (2pt);
            \draw[fill=black] (9,0) circle (2pt);
            \draw[fill=black] (10,0) circle (2pt);
            \draw[fill=black] (11,0) circle (2pt);
            \draw (8,0)--(11,0);
            \node at (9.5, 0.2) {$4$};
            \node at (7.5,0) {$F_4$};

            \draw[fill=black] (8,-1.5) circle (2pt);
            \draw[fill=black] (9,-1.5) circle (2pt);
            \draw (8,-1.5)--(9,-1.5);
            \node at (8.5,-1.3)  {$6$};
            \node at (7.5,-1.5) {$G_2$};

            \draw[fill=black] (8,-3) circle (2pt);
            \draw[fill=black] (9,-3) circle (2pt);
            \draw[fill=black] (10,-3) circle (2pt);
            \draw (8,-3)--(10,-3);
            \node at (9.5, -2.8){$5$};
            \node at (7.5,-3) {$H_3$};

            \draw[fill=black] (8,-4.5) circle (2pt);
            \draw[fill=black] (9,-4.5) circle (2pt);
            \draw[fill=black] (10,-4.5) circle (2pt);
            \draw[fill=black] (11,-4.5) circle (2pt);
            \draw (8,-4.5)--(11, -4.5);
            \node at (10.5, -4.3){$5$};
            \node at (7.5,-4.5) {$H_4$};

            \draw[fill=black] (8,-6) circle (2pt);
            \draw[fill=black] (9,-6) circle (2pt);
            \draw (8,-6)--(9,-6);
            \node at (8.5, -5.8){$m$};
            \node at (7.3,-6) {$I_2(m)$};
            \node at (10.8, -6){, $m=5$ or $m\geq7$};
				
		\end{tikzpicture}
	\caption{Coxeter graphs of type $X_n$ where $|V(X_n)|=n$.}
	\end{center}
\end{figure}
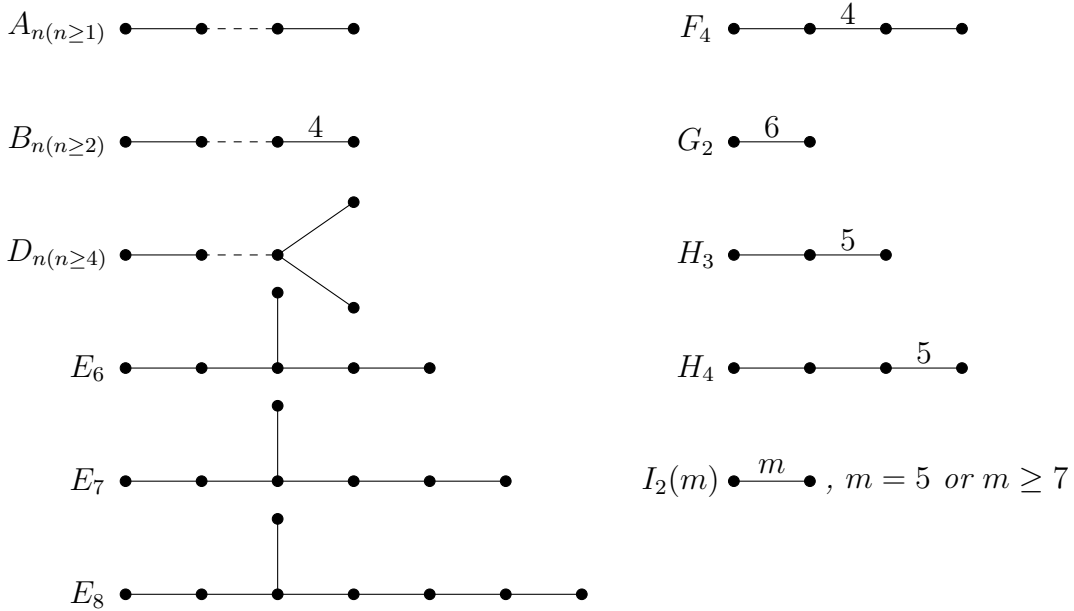

\item The group $W_\Gamma$ is virtually free abelian of positive rank if and only if every connected component of $\Gamma$ is isomorphic to one of the graphs in Figure 1 or in Figure 2 and there exists at least one connected component which is isomorphic to a graph in Figure 1.
\end{enumerate}
\end{theorem}

Many group theoretical properties of a Coxeter group $W_\Gamma$ can be easily read off the combinatorial structure of $\Gamma$ or a modification of $\Gamma$. Let $\Gamma_{odd}\subseteq\Gamma $ be the subgraph obtained from $\Gamma$ by removing all edges with even or $\infty$ labels. The abelianization of $W_\Gamma$, denoted by $W^{ab}_\Gamma$ is visible in the structure of $\Gamma_{odd}$.
\begin{proposition}(\cite[Fact 3.16]{Muhlherr})
\label{Abel}
Let $W_\Gamma$ be a Coxeter group. The abelianization of $W_\Gamma$ is isomorphic to $(\mathbb{Z}/2\mathbb{Z})^k$ where $k$ is the number of connected components of $\Gamma_{odd}$. 
\end{proposition}

Let us take a closer look at the Coxeter groups $W_\Gamma$ where $\Gamma$ is isomorphic to a graph in Figure 1. These groups and their abelianizations admit a semidirect product structure as follows.
\begin{proposition} (\cite[§4.2]{Humphreys})
\label{IrrAffineAb}
 Let $W_\Gamma$ be a Coxeter group. If $\Gamma$ is isomorphic to one of the graphs in Figure 1, then $W_\Gamma\cong\mathbb{Z}^{n-1}\rtimes W_\Omega$ where $n=|V(\Gamma)|$ and $\Omega$ is one of the types $A_n, B_n, E_6, E_7, E_8$, $F_4$, $G_2$. In particular,
\begin{enumerate}
\item $W_{\widetilde{A}_1}\cong\mathbb{Z}\rtimes\mathbb{Z}/2\mathbb{Z}$ and $W_{\widetilde{A}_1}^{ab}\cong\mathbb{Z}/2\mathbb{Z}\times\mathbb{Z}/2\mathbb{Z}$
\item $W_{\widetilde{A}_n}\cong\mathbb{Z}^{n}\rtimes W_{A_n}$ and $W^{ab}_{\widetilde{A}_n}\cong \mathbb{Z}/2\mathbb{Z}$
\item $W_{\widetilde{B}_n}\cong\mathbb{Z}^{n}\rtimes W_{B_n}$ and $W^{ab}_{\widetilde{B}_n}\cong\mathbb{Z}/2\mathbb{Z}\times \Z/2\Z$
\item $W_{\widetilde{C}_n}\cong\mathbb{Z}^{n}\rtimes W_{B_n}$ and $W^{ab}_{\widetilde{C}_n}\cong (\Z/2\Z)^3$
\item $W_{\widetilde{D}_n}\cong\mathbb{Z}^{n}\rtimes W_{D_n}$ and $W^{ab}_{\widetilde{D}_n}\cong \Z/2\Z$
\item $W_{\widetilde{E}_6}\cong\mathbb{Z}^{6}\rtimes W_{E_6}$ and $W^{ab}_{\widetilde{E}_6}\cong \Z/2\Z$
\item $W_{\widetilde{E}_7}\cong\mathbb{Z}^{7}\rtimes W_{E_7}$ and $W^{ab}_{\widetilde{E}_7}\cong\mathbb{Z}/2\mathbb{Z}$
\item $W_{\widetilde{E}_8}\cong\mathbb{Z}^{8}\rtimes W_{E_8}$ and $W^{ab}_{\widetilde{E}_8}\cong\mathbb{Z}/2\mathbb{Z}$
\item $W_{\widetilde{F}_4}\cong\mathbb{Z}^{4}\rtimes W_{F_4}$ and $W^{ab}_{\widetilde{F}_4}\cong \mathbb{Z}/2\mathbb{Z}\times\mathbb{Z}/2\mathbb{Z}$
\item $W_{\widetilde{G}_2}\cong\mathbb{Z}^{2}\rtimes W_{G_2}$ and $W^{ab}_{\widetilde{G}_2}\cong \mathbb{Z}/2\mathbb{Z}\times\mathbb{Z}/2\mathbb{Z}$
\end{enumerate}
\end{proposition}

The orders of the finite parts in these semidirect decompositions of the irreducible affine Coxeter groups have been computed in \cite[§2.10]{Humphreys}.

\begin{center}
\begin{tabular}[h]{|c|c|c|c|c|c|c|c|c|}
\hline
$\Gamma$ & $A_n$ & $B_n$ & $D_n$ & $E_6$ & $E_7$ & $E_8$ & $F_4$ & $G_2$\\
\hline
$|W_\Gamma|$ & $(n+1)!$ & $2^n\cdot n!$ & $2^{n-1}n!$  & $2^73^45$ & $2^{10}3^4\cdot 5\cdot 7$ & $2^{14}3^55^27$ & $2^73^2$ & $12$\\
\hline
\end{tabular}
\end{center}

As an immediate corollary we obtain
\begin{corollary}
\label{FiniteOrder}
Let $X_n$ and $Y_n$ be Coxeter graphs with $n\geq 4$ vertices. Assume that $X_n$ and $Y_n$ are isomorphic to graphs of type $A_n, B_n, D_n, E_6, E_7, E_8$ or $F_4$. If $|W_{X_n}|=|W_{Y_n}|$, then $X_n\cong Y_n$. 
\end{corollary}

Here we are in particular interested in the structure of virtually abelian Coxeter groups. The next lemma will needed later on.

\begin{lemma}\label{factoriso}
    Let $A_1,A_2$ be abelian torsion free groups and $F_1,F_2$ finite groups. If $A_1\rtimes F_1\cong A_2\rtimes F_2$, then $F_1\cong F_2$. 
    %If $A_i$ is finitely generated we also obtain $A_1\cong A_2$.
    Furthermore, there exist injective group homomorphisms $\varphi_1\colon A_1\hookrightarrow A_2$ and $\varphi_2\colon A_2\hookrightarrow A_1$. In particular, if $A_1$ and $A_2$ are finitely generated, then $rank(A_1)=rank(A_2)$.
\end{lemma}
\begin{proof}
Let $f\colon A_1\rtimes F_1\to A_2\rtimes F_2$ denote an isomorphism.

We have the canonical projections $\pi_i\colon A_i\rtimes F_i\to F_i$. Since $A_i$ is torsion free, the kernel of $\pi_i$ is torsion free, so given any finite subgroup $H$ of $A_i\rtimes F_i$, the restriction of $\pi_i$ to $H$ is injective. We obtain:
$$|\pi_2\circ f(\{1\}\rtimes F_1)|\leq |F_2|\quad \text{and} \quad |\pi_1\circ f^{-1}(\{1\}\rtimes F_2)|\leq |F_1|$$
Hence $|F_1|=|F_2|$. Moreover $F_1\cong (\{1\}\rtimes F_1)$ and due to the number of elements being equal, $\pi_2\circ f(\{1\}\rtimes F_1)=F_2$ which means $\pi_2 \circ f$ restricted to $\{1\}\rtimes F_1$ is an isomorphism, which is what we had to show.

Let $k$ denote the order of the group $F_1$. We define $\varphi_1\colon A_1\to A_2\rtimes F_2$ by $\varphi_1(a):=f(a)^k$. Since $A_1$ is abelian and torsion free, it follows that $\varphi_1$ is an injective group homomorphism because of the following reason: $f$ is an isomorphism, hence $f(A_1)$ is also torsion free. So if $f(a)^k=1$, then $f(a)=1$ and again since $f$ is an isomorphism we obtain $a=1$. Furthermore, $\varphi_1(A_1)\subseteq A_2\times\left\{1\right\}$.
The map $\varphi_2\colon A_2\to A_1\rtimes F_1$ is defined in the same way, $\varphi_2(a)=(f^{-1}(a))^k$. The map $\varphi_2$ is also an injective group homomorphism with $\varphi_2(A_2)\subseteq A_1\times\left\{1\right\}$.

The last statement of the lemma follows from the fact that if $\Z^n\hookrightarrow \Z^m$, then $n\leq m$.
\end{proof}

As a corollary we obtain
\begin{corollary}
\label{AffineCoxeterGroups}
Let $W_\Gamma$ and $W_\Delta$ be Coxeter groups where $\Gamma$ and $\Delta$ are graphs in Figure 1 and  $W_\Gamma\cong\mathbb{Z}^n\rtimes W_{\Omega_1}$, $W_\Delta\cong\mathbb{Z}^m\rtimes W_{\Omega_2}$ be the corresponding semidirect decompositions from Proposition \ref{IrrAffineAb}. Then the following statements are equivalent
\begin{enumerate}
\item[(i)] $W_\Gamma\cong W_\Delta$
\item[(ii)] $n=m$, $|W_{\Omega_1}|=|W_{\Omega_2}|$ and $|W_\Gamma^{ab}|=|W_\Delta^{ab}|$
\item[(iii)] $\Gamma\cong\Delta$
\end{enumerate}
\end{corollary}
\begin{proof}
First, we assume that $W_\Gamma\cong W_\Delta$. Since isomorphic groups have isomorphic abelianizations we know that $|W_\Gamma^{ab}|=|W_\Delta^{ab}|$. By Lemma \ref{factoriso} it follows that $n=m$ and $|W_{\Omega_1}|=|W_{\Omega_2}|$. This proves $(i)\Rightarrow (ii)$.

Now we prove $(ii)\Rightarrow (iii)$. We have to consider several cases. If $n=1$, then $\Gamma\cong\Delta\cong\widetilde{A_1}$.
 If $n=2$, then we have three possibilities for $W_\Gamma$: $W_\Gamma\cong W_{\widetilde{G_2}}$ or $W_\Gamma\cong W_{\widetilde{A_2}}$ or $W_\Gamma\cong W_{\widetilde{C_2}}$. Since $|W_{G_2}|=12$, $|W_{A_2}|=6$ and $|W_{B_2}|=8$, we obtain $\Gamma\cong\Delta$. If $n=3$, then we have $3$ possibilities  for $W_\Gamma$: $W_\Gamma\cong W_{\widetilde{A_3}}$ or $W_\Gamma\cong W_{\widetilde{B_3}}$ or $W_\Gamma\cong W_{\widetilde{C_3}}$. Since the abelianizations of these groups are not isomorphic we conclude that $\Gamma\cong\Delta$. 
 If $n\geq 4$, then  by Corollary \ref{FiniteOrder} it follows that $\Omega_1\cong\Omega_2$. Assume for contradiction that $\Gamma$ is not isomorphic to $\Delta$. Then $\Gamma\cong\widetilde{B_n}$ and $\Delta\cong\widetilde{C_n}$. In particular $|W^{ab}_\Gamma|\neq |W^{ab}_\Delta|$, a contradiction.
 
 The last statement $(iii)\Rightarrow (i)$ is obvious.
\end{proof}

Corollary \ref{AffineCoxeterGroups} is related to the notion of rigidity of Coxeter groups. A Coxeter group $W_\Gamma$ is called \emph{rigid} if whenever $W_\Gamma\cong W_\Delta$ then $\Gamma\cong\Delta$. It was proven in \cite{Charney} that irreducible affine Coxeter groups are rigid.

By definition, a group $G$ is called \emph{directly decomposable} if there exist non-trivial groups $A$ and $B$ such that $G\cong A\times B$, otherwise $G$ is called \emph{directly indecomposable}. 

\begin{theorem}(\cite[Theorem 2.17 and Theorem 3.3]{Nuida})
\label{Nuida}
Let $W_\Gamma$ be an irreducible Coxeter group. The group $W_\Gamma$ is directly indecomposable if and only if $\Gamma$ is not isomorphic to $B_{2k+1}$, $I_2(4k+2)$, $k\geq 1$, $E_7$, $H_3$ from Figure 2. (Note that in our notation $I_2(6)=G_2$).

Moreover, $W_\Gamma$ is directly decomposable in products of Coxeter groups if and only if $\Gamma$ is of type $B_{2k+1}$ or $I_2(4k+2)$, $k\geq 1$ and the direct decomposition of these groups is as follows: $W_{B_{2k+1}}\cong\Z/2\Z\times W_{D_{2k+1}}$, $W_{I_2(4k+2)}\cong\Z/2\Z\times W_{I_2(2k+1)}$.
\end{theorem}

\section{Just infinity}
In the study of infinite groups one is interested in understanding all non-trivial normal subgroups. Sometimes for an infinite group $G$ every non-trivial normal subgroup has finite index. This is the notion of just infinity which was introduced by Baumslag.
\begin{definition}
A group $G$ is called \emph{just infinite} if $G$ is infinite and all non-trivial normal subgroups of $G$ are of finite index.
\end{definition}
In this chapter we discuss Coxeter groups in relation to the property of being just infinite. Since we are in particular interested in quotients of Coxeter groups, we discuss some explicit methods of constructing a quotient of a Coxeter group, especially ones giving us nice quotients, i.e. again Coxeter groups.

\begin{proposition}(\cite[Lemma 3.6]{Cooper})
    Let $W_\Gamma$ be a Coxeter group. Suppose there are two vertices $s,t\in V(\Gamma)$ such that the label of the edge $\left\{s,t\right\}$ is $\infty$. Then either $W_\Gamma$ is virtually $\mathbb{Z}$ or there exists a surjection $\pi\colon W_\Gamma\twoheadrightarrow W_\Omega$, where $W_\Omega$ infinite and $V(\Omega)=V(\Gamma)$ and $E(\Omega)=E(\Gamma)\cup\left\{\left\{s,t\right\}\right\}$ and the edge label of $\left\{s,t\right\}$ is any integer $\geq 7$.
\end{proposition}
So if the graph $\Gamma$ used to define the Coxeter group has a label $\infty$, then either $W_\Gamma$ is already virtually $\mathbb{Z}$ or we can explicitly construct an infinite quotient.

Furthermore a very similar idea works in case there are non-prime edge-labels.
\begin{proposition}(\cite[§3]{Mihalik})
    Let $W_\Gamma$ denote a Coxeter group and suppose there is an edge label $m(\left\{v,w\right\})$ which is not prime. Let $\Gamma'$ denote the graph obtained from $\Gamma$ by replacing the edge label $m(\left\{v,w\right\})$ by one of its prime factors. Then we obtain a surjection $\pi\colon W_\Gamma \twoheadrightarrow W_{\Gamma'}$.
\end{proposition}

For an additional but related way to obtain a quotient isomorphic to a Coxeter group we recall the definition of special parabolic subgroups of a Coxeter group. Given a Coxeter group $W_\Gamma$ and $X\subseteq V(\Gamma)$, the subgroup generated by $X$ is called a \emph{special parabolic subgroup} and it is canonically isomorphic to $W_\Omega$ where $\Omega$ is a full subgraph of $\Gamma$ with the vertex set $X$. 
\begin{proposition}(\cite[Proposition 2.1]{Gal})
Let $W_\Gamma$ be a Coxeter group and $W_\Omega$ be a special parabolic subgroup. If for every $\left\{v, w\right\}\in E(\Gamma)$ where $v\in V(\Omega), w\in V(\Gamma)-V(\Omega)$ the edge label is even or $\infty$, then
 there exists a surjective group homomorphism $\varphi\colon W_\Gamma \twoheadrightarrow W_\Omega$.
\end{proposition}

For the next construction, it is a lot easier to work with a different convention for the defining graph $\Gamma$ of the Coxeter group. Let $\Delta$ be the graph obtained from $\Gamma$ by removing all edges labeled infinity and by adding an edge with label $2$ between all vertices that were previously not connected by an edge.
\begin{proposition}
    Let $W_\Delta$ be a Coxeter group. Let $p\in \mathbb{P}-\{2\}$ denote a prime number. We denote by $\Delta^{(p)}$ the graph obtained from $\Delta$ by removing all edges labeled $p$. If $\Delta^{(p)}$ has three or more connected components, then there exists a projection $\pi\colon W_\Delta\to W_{\Delta'}$, where $W_{\Delta'}$ is an infinite Coxeter group with $|V(\Delta')|=3$.
\end{proposition}
\begin{proof}
    Let $\Delta'$ denote the complete graph with vertex set $V'=\{v_1',v_2',v_3'\}$ and every edge label equal to $p$. Let $C_1,C_2,...,C_n$ denote the connected components of $\Delta^{(p)}$, $n\geq 3$. Then we set $\pi\colon W_\Delta\to W_{\Delta'}$, $\pi(v_i)=\begin{cases}
    &v_i'\quad \textrm{if }i\in\{1,2,3\},\\
    &1\quad \textrm{ else.}
    \end{cases}$\\
    It is easy to check that this is indeed a homomorphism. Furthermore $W_{\Delta'}$ is infinite due to the classification of finite Coxeter groups (since $p\geq 3$).
\end{proof}

We note that it is not possible with the methods above to constuct an infinite proper quotient of the Coxeter group $W_\Gamma$ where $\Gamma$ is isomorphic to the graph in Figure 3.
\begin{figure}[h]
	\begin{center}
	\captionsetup{justification=centering}
		\begin{tikzpicture}
			\draw[fill=black]  (0,0) circle (2pt);
            \draw[fill=black]  (1,0) circle (2pt);
            \draw[fill=black]  (0.5,0.7) circle (2pt);
            \draw (0,0)--(1,0);
            \draw (1,0)--(0.5, 0.7);
            \draw (0,0)--(0.5, 0.7);
            \node at (0.5, -0.25) {$5$};
            \node at (0.1, 0.45) {$5$};
            \node at (0.9, 0.45) {$5$};
        \end{tikzpicture}
        \caption{}
	\end{center}
\end{figure}

Hence, we need a different method to show that the Coxeter group with the defining graph in Figure 3 is not just infinite.

\vspace{0.5cm}

Let $G$ be an infinite group and $H\subseteq G$ be a non-trivial subgroup. 
By definition, the \emph{core} of $H$, denoted by ${\rm core}_G(H)$ is the maximal subgroup of $H$ that is a normal in $G$. Indeed ${\rm core}_G(H)=\bigcap_{g\in G}gHg^{-1}$. If ${\rm core}_G(H)=\left\{1\right\}$, then $H$ is called \emph{core-free}. We are interested in properties of $H$ that ensure $H$ is not core-free. It is known that if $H$ has finite index in $G$, then $H$ has a non-trivial core since ${\rm core}_G(H)$ also has finite index in $G$.  

Since our goal here is to obtain infinite quotients of a given group we are in particular interested in the subgroups of infinite index which have a non-trivial core. 
 
We recall that a \emph{subdirect product} of $\prod_{i\in I}G_i$ is a subgroup
$G\subseteq \prod_{i\in I}G_i$ which projects surjectively onto each factor. 

\begin{lemma}
Let $G$ be a group and for $i\in I$ let $N_i\trianglelefteq G$ be a normal subgroup. Then $G/\bigcap_{i\in I} N_i$ is isomorphic a subdirect product of $\prod_{i\in I}G/N_i$.
\end{lemma}
\begin{proof}
We define a map $\iota\colon G/\bigcap_{i\in I} N_i\to \prod_{i\in I}G/N_i$, $gN\mapsto (gN_1,gN_2,\ldots)$. Note that this map is an injective homomorphism, hence $G/\bigcap_{i\in I} N_i$ is isomorphic to its image. So we only need to show that all the projections $\pi_i\colon \iota(G/\bigcap_{i\in I} N_i)\to G/N_i$ are surjective. But this follows immediately from the definition of $\iota$.
\end{proof}

\begin{proposition}
\label{CoreNonTrivial}
Let $G$ be a group and $H\subseteq G$ be a subgroup. Assume that there exists a normal subgroup $N\trianglelefteq G$ such that $H\trianglelefteq N$ and $N/H$ has a property $\mathcal{P}$ which is preserved by all subdirect products of $\prod_{I}N/H$. If $N$ does not have $\mathcal{P}$, then ${\rm core}_G(H)$ in non-trivial.
\end{proposition}
\begin{proof}
First, we note that since $N$ is normal in $G$, for any $g\in G$ the subgroup $gHg^{-1}$ is also normal in $N$. More precisely: there exists a homomorphism $\varphi\colon N\twoheadrightarrow N/H$ such that $\ker(\varphi)=H$. Let $\gamma_{g^{-1}}\colon N\rightarrow N$ be the conjugation map, so $n\mapsto g^{-1}ng$. Note that  $g^{-1}ng\in N$ since $N$ is normal in $G$, thus the map $\gamma_{g^{-1}}$ is well-defined. We have $\ker(\varphi\circ\gamma_{g^{-1}})=gHg^{-1}$
We define $L_g:=N/gHg_i^{-1}$. Note that $L_g\cong N/H$.
 The group $N/\bigcap_{g\in G} gHg^{-1}$ is a subdirect product of $\prod_{g\in G}L_g$ by the previous Lemma. By assumption $N$ has property $\mathcal{P}$ and $N/core_G(H)$
does not have $\mathcal{P}$, thus $core_G(H)$ has to be non-trivial.
\end{proof}

By definition a group $G$ is \emph{virtually indicable}  if $G$ has a  subgroup $H$ of finite index and $H$ has an infinite abelian cyclic group as a quotient.

\begin{remark}
Note that if $G$ is virtually indicable, then there exists a finite-index normal subgroup $N\trianglelefteq G$ which has an infinite abelian cyclic group as a quotient. 
\end{remark}
\begin{proof}If $G$ is virtually indicable then $G$ has a subgroup $H$ of finite index such that $H/K\cong\mathbb{Z}$, where $\pi\colon H\twoheadrightarrow\mathbb{Z}$ and $K=\ker(\pi)$. Then
$core_G(H)$ has finite index in $G$ and hence in $H$, since if this was not the case, we would have ${\rm core}_G(H)\subseteq K$ and because $K$ has infinite index in $H$, so would ${\rm core}_G(H)$, which is impossible. Thus $\pi({\rm core}_G(H))$ is a non-trivial subgroup of $\mathbb{Z}$ and hence ${\rm core}_G(H)$ projects onto $\mathbb{Z}$.
\end{proof} 

\begin{corollary}
\label{IndicableAbelian}
Let $G$ be a group. If $G$ is virtually indicable  and is not virtually abelian, then $G$ is not just infinite.
\end{corollary}
\begin{proof}
Since $G$ is virtually indicable, there exists a normal subgroup $N$ of finite index in $G$ and a surjective group homomorphism $\varphi\colon N\twoheadrightarrow\Z$. Thus $N/\ker(\varphi)\cong \Z$. Let $\mathcal{P}$ be the property to be abelian. Note that this property is preserved by all subdirect products of $\prod_I N/ker(\varphi)$. By assumption, $G$ is not virtually abelian, hence $N$ does not have $\mathcal{P}$. By Proposition \ref{CoreNonTrivial} follows that $core_G(ker(\varphi))$ is non-trivial. Moreover, since $\ker(\varphi)$ has infinite index in $G$, we know that $core_G(\ker(\varphi))$ is also of infinite index in $G$.
\end{proof}

Whether a Coxeter group is virtually indicable can be easily decided by looking at the defining graph. 

\begin{theorem}(\cite[Theorem 1.1]{Cooper})
\label{Indicable}
Let $W_\Gamma$ be a Coxeter group. If $W_\Gamma$ is infinite, then $W_\Gamma$ is virtually indicable. In particular, $W_\Gamma$ is virtually indicable if and only if there exists at least one connected component of $\Gamma$ which is not isomorphic to a graph in Figure 2.
\end{theorem}

With these tools at hand, we can now prove Theorem 1.1.

\begin{proof}[Proof of Theorem 1.1]
Let $W_\Gamma$ be an infinite Coxeter group. 
If $\Gamma$ is one of the graphs in Figure 1, then it was proven in \cite[Corollary 0.3]{Maxwell} (see also \cite[Proposition 3.2]{ParisVarghese}) that every proper normal subgroup in $W_\Gamma$ is of finite index. Hence $W_\Gamma$ is just infinite. 

Assume now that $\Gamma$ is not isomorphic to one of the graphs in Figure 1. If $\Gamma$ is not connected, then $W_\Gamma$ is a direct product of $W_{\Gamma_1},\ldots, W_{\Gamma_n}$ where $\Gamma_1\ldots,\Gamma_n$ are connected components of $\Gamma$. Since $W_\Gamma$ is infinite, there exists $j\in\left\{1,\ldots,n\right\}$ such that $W_{\Gamma_j}$ is infinite, thus the kernel of the canonical projection $W_\Gamma\to W_{\Gamma_j}$ yields a proper normal subgroup of infinite index in $W_\Gamma$, hence $W_\Gamma$ is not just infinite. If $\Gamma$ is connected, then by Theorem \ref{FiniteVirtuallyAbelian} $W_\Gamma$ is not virtually abelian and by Theorem \ref{Indicable} it is virtually indicable, hence  Corollary \ref{IndicableAbelian} tells us that $W_\Gamma$ is not just infinite.
\end{proof}

\section{Profinite rigidity}
A group $G$ is said to be \emph{residually finite}  if for each non-trivial $g\in G$ there exists a finite group $F$ and a group homomorphism $\varphi\colon G\to F$ such that $\varphi(g)\neq 1$. Equivalently, a group $G$ is residually finite if the intersection of all its finite index normal subgroups is trivial.

Unless otherwise stated, all discrete groups, which we consider here, will be \emph{finitely generated and residually finite}. In particular, if $G$ is such a group, then $G$ admits many finite quotients. We denote by $\mathcal{F}(G)$ the set of isomorphism classes of finite quotients of a group $G$. Since our main protagonists in this article are Coxeter groups, which are known to be residually finite, we ask if it is possible to distinguish Coxeter groups by their finite quotients. 
%We conjecture:
\begin{question}
Let $W_\Gamma$ and $W_\Omega$ are Coxeter groups. If $\mathcal{F}(W_\Gamma)=\mathcal{F}(W_\Omega)$, does it follow that $W_\Gamma\cong W_\Omega$?
\end{question}

It was proven in \cite{Dixon} that the study of the set of isomorphism classes of finite quotients of a group $G$ can be reformulated in the world of profinite groups. By definition, a \emph{profinite group} is a totally disconnected compact Hausdorff group. A profinite group $G$ is \emph{finitely generated} if there exists a finite set $X\subseteq G$ such that $\overline{\langle X\rangle}=G$. 

Let $G$ be a residually finite group and $\mathcal{N}$ be the set of all finite index normal subgroups of $G$. We equip each $G/N$, $N\in\mathcal{N}$ with the discrete topology, then $\prod_{N\in\mathcal{N}} G/N$ equipped with the product topology is a profinite group. We define a map $i\colon G\to \prod_{N\in\mathcal{N}} G/N$ by $g\mapsto (gN)_{N\in\mathcal{N}}$. Note, that $i$ is injective, since $G$ is residually finite. The \emph{profinite completition} $\widehat{G}$ of $G$ is defined as $\widehat{G}:=\overline{i(G)}$. Clearly, if $G$ is finite, then $\widehat{G}=G$. Moreover, we endow $\left\{G/N\mid N\in\mathcal{N}\right\}$ with the structure of a directed poset by setting $G/N\leq G/M$ whenever $M\subseteq N$ for $M,N\in\mathcal{N}$. For $G/N\leq G/M$ we denote by $\pi_{NM}\colon G/M\rightarrow G/N$ the canonical projection. Thus $((G/N)_{N\in\mathcal{N}}, (\pi_{NM})_{N,M\in\mathcal{N}})$ is an inverse system. One can show that $\widehat{G}$ is isomorphic to the inverse limit  of this inverse system. 
For more information on profinite completions of groups we refer to \cite{Reid}.

\begin{proposition}(\cite{Dixon})
Let $G$ and $H$ be finitely generated residually finite groups. Then $\mathcal{F}(G)=\mathcal{F}(H)$ if and only if $\widehat{G}\cong\widehat{H}$.
\end{proposition}

\begin{definition}
Let $\mathcal{C}$ be a class of finitely generated residually finite groups. A group $G$ is called \emph{profinitely rigid among $\mathcal{C}$} if $G\in \mathcal{C}$ and for any group $H\in\mathcal{C}$ whenever $\widehat{G}\cong\widehat{H}$, then $G\cong H$. 

In particular, if $\mathcal{C}$ is the class consisting of all finitely generated residully finite groups and $G$ is profinitely rigid among $\mathcal{C}$, then we call $G$ \emph{profinitely rigid (in the absolute sense)}.
\end{definition}

We need the following general fact about profinite completions which was proven in \cite[Remark 3.2]{Reid}.
\begin{lemma}\label{abelianization}
    Let $G$ and $H$ be finitely generated, residually finite groups. If $\widehat{G}\cong \widehat {H}$, then $G^{ab}\cong H^{ab}$.
\end{lemma}

As a corollary we obtain a profinite rigidity invariant among all Coxeter groups which we can easily read off the Coxeter graphs.
\begin{corollary}
Let $W_\Gamma$ and $W_\Delta$ be Coxeter groups. If $\widehat{W_\Gamma}\cong\widehat{W_\Delta}$, then $|\pi_0(\Gamma_{odd})|=|\pi_0(\Delta_{odd})|$ where we denote by $\pi_0(X)$ the set of connected components of a graph $X$.
\end{corollary}
\begin{proof}
The conclusion of the corollary follows immediately combining Proposition \ref{Abel} with Lemma \ref{abelianization}. 
\end{proof}

Given two groups $G$ and $H$, the profinite completions $\widehat{G\times H}$ resp. $\widehat{G\rtimes H}$  are related to $\widehat{G}\times\widehat{H}$ resp. $\widehat{G}\rtimes\widehat{H}$ as the following lemma shows. 

\begin{lemma}(\cite[Proposition 2.6]{Grunewald})
\label{profiniteproducts}
   Let $G$, $H$ be finitely generated residually finite groups and let 
   $\varphi\colon H\to \Aut(G)$ be a homomorphism. The inclusions $G\hookrightarrow \widehat{G}$, $H\hookrightarrow\widehat{H}$ induce an isomorphism $\widehat{G\rtimes_\varphi H}\cong \widehat{G}\rtimes_{\widehat\varphi} \widehat{H}.$
\end{lemma}

It is natural to ask if the rank of a virtually free abelian group $G$ is determined by $\mathcal{F}(G)$ or $\widehat{G}$ respectively. To prove that this is indeed the case for semidirect products we need the following proposition.

\begin{proposition}
\label{rank}
Let $n, m\in\mathbb{N}_{\geq 1}$. If $\varphi\colon \widehat{\Z}^n\to \widehat{\Z}^m$ is a group homomorphism (not necessarily continuous), then surjectivity of $\varphi$ implies $m\leq n$ and injectivity of $\varphi$ implies $m\geq n$.
\end{proposition}
\begin{proof}
Recall that every abelian group has a canonical $\Z$-modul structure and every group homomorphism between abelian groups is $\Z$-linear. In the same way, every abelian profinite group has a canonical $\widehat{\Z}$-modul structure, since $\mathbb{Z}$ is dense in $\widehat{\Z}$ and the group operations are continuous.
Note, that for the same reason every continuous group homomorphism between abelian profinite groups is $\widehat{\Z}$-linear.

Further, it was proven in \cite[Theorem 1.1]{NikolovSegal} that every group homomorphism between finitely generated profinite groups is continuous, hence $\varphi$ is $\widehat{\Z}$-linear. 
Thus applying \cite[Corollary 5.11]{Conrad} we conclude that if $\varphi$ is injective, then $n\leq m$ and if $\varphi$ is surjective then $n\geq m$.
\end{proof}

Now we are able to prove
\begin{lemma}
\label{VirAbelian}
Let $G_1$ and $G_2$ be finitely generated virtually free abelian groups.
\begin{enumerate}
\item[(i)]  Let $N\trianglelefteq G_1$ be a normal finite index subgroup where $N\cong\Z^n$. Assume that  $G_2\cong\Z^m\rtimes F_2$ for $F_2$ finite.
If $\widehat{G_1}\cong\widehat{G_2}$, then $n\leq m$.
\item[(ii)] Assume that $G_1\cong\Z^n\rtimes F_1$ and $G_2\cong\Z^m\rtimes F_2$  where $F_1$, $F_2$ are finite groups. If $\widehat{G_1}\cong\widehat{G_2}$, then $F_1\cong F_2$ and $n=m$.
\end{enumerate}
\end{lemma}
\begin{proof}
We start by proving (i). By Lemma \ref{profiniteproducts} the profinite completion of $G_2$ is isomorphic to $\widehat{\Z}^m\rtimes F_2$. By assumption we have
$$\widehat{G_1}\cong\widehat{G_2}\cong \widehat{\Z}^m\rtimes F_1$$
Since $N\cong\Z^n$ has finite index in $G_1$ it follows from
\cite[Corollary 2.8]{BridsonConder} that $\widehat{N}\hookrightarrow\widehat{G_1}$.
In particular, we obtain $\varphi\colon\widehat{\Z}^n\hookrightarrow \widehat{\Z}^m\rtimes F_1$. Let $k=|F_2|$, then the map $\psi\colon \widehat{\Z}^n\to\widehat{\Z}^m\rtimes F_1$ defined by $\psi(a)=(\varphi(a))^k$ is an injective group homomorphism with image contained in $\widehat{\Z}^m\rtimes \left\{1\right\}$ (see the proof of Lemma \ref{factoriso}). It follows by Proposition \ref{rank} that $n\leq m$.

For the proof of (ii) we first note that the profinite completion of $G_1$ resp. $G_2$ is isomorphic to $\widehat{\Z}^n\rtimes F_1$ resp. $\widehat{\Z}^m\rtimes F_2$ (see Lemma \ref{profiniteproducts}). We apply Lemma \ref{factoriso} to conclude that $F_1\cong F_2$ and that there exist injective group homomorphisms $\varphi_1\colon\widehat{\Z}^n\to\widehat{\Z}^m$ and 
$\varphi_2\colon\widehat{\Z}^m\to\widehat{\Z}^n$. It follows that $n=m$ by Proposition \ref{rank}.
\end{proof}

\section{Proof of Theorem 1.2}
We start with the proof of two small but useful lemmata.

\begin{lemma}\label{product}
    Let $G_1=A_1\rtimes_{\varphi_1} F_1$ and $G_2=A_2\rtimes_{\varphi_2} F_2$. Then $G_1\times G_2\cong (A_1\times A_2)\rtimes_{\varphi_1\times \varphi_2} (F_1\times F_2)$, where $\varphi_1\times \varphi_2$ is the image of $(\varphi_1,\varphi_2)$ under the natural map $F_1\times F_2\to \Aut(A_1\times A_2)$.
\end{lemma}
\begin{proof}
    This is an easy exercise, the natural bijective map between the two sets is a homomorphism due to the definition of $\varphi_1\times \varphi_2$. 
\end{proof}
Another ingredient for the proof will be the following observation
\begin{lemma}\label{NoFiniteNormal}
    Let $G=A\rtimes F$ for a torsion free abelian group $A$ and $F$ a finite group. Assume that $G$ is finitely generated and residually finite. If $G$ does not have non-trivial finite normal subgroups, then $\widehat{G}$ does not have non-trivial finite normal subgroups either.
\end{lemma}
\begin{proof}
The profinite completion of $\widehat{G}$ is isomorphic to $\widehat{A}\rtimes_\varphi F$ (see Lemma \ref{profiniteproducts}). Now suppose there exists a finite normal subgroup $N$ in $\widehat{G}$. We have the natural projection $\pi\colon\widehat{G}\to F$. Moreover, $\pi(N)\neq \{1\}$, since $\widehat{A}$ is torsion free by \cite[Proposition 2.1]{Kropholler} and $N$ is finite.

There are two main steps in the proof which are done by explicit calculations:\\
\underline{Step 1:} $\varphi(\pi(N))=\{id\}$.

Note that $\pi_{\mid_N}$ is injective since $\widehat{A}$ is torsion free. Let $n\in N$ with $n=(a,f)$ and $b\in \widehat{A}$ denote arbitrary elements. Then we obtain
$$(b,1)\cdot (a,f)\cdot (b,1)^{-1} = (ba\varphi(f)(b^{-1}),f)$$
Since the restriction of $\pi$ to $N$ is injective, we obtain $a=ba\varphi(f)(b^{-1})$ and since $\widehat{A}$ is abelian, $\varphi(f)(b)=b$, hence completing step 1.

\underline{Step 2:} $\pi(N)$ is normal in $\widehat{G}$.

Let $g\in \widehat{G}$ denote an arbitrary element. We write $g=(b,f_1)$ and let $(1,f)\in \pi(N)$. We obtain
$$(b,f_1)(1,f)(b,f_1)^{-1}=(b\varphi(f_1ff_1^{-1})(b^{-1}),f_1ff_1^{-1})$$
Since $f_1ff_1^{-1}\in \pi(N)$, we conclude by Step 1, that $\varphi(f_1ff_1^{-1})(b^{-1})=b^{-1}$ and hence
$$g(1,f)g^{-1}=(bb^{-1},f_1ff_1^{-1})=(1,f_1ff_1^{-1})\in \pi(N)$$

Since $G$ embeds in $\widehat{G}$, and $\pi(N)\subseteq F$, we conclude, that $\pi(N)$ is also normal in $G$. Since $N$ is finite and $\pi(N)$ is non-trivial, we have reached a contradiction.
\end{proof}

We can now prove Theorem 1.2.
\begin{proof}[Proof of Theorem 1.2]
Let $W_\Gamma$ be a just infinite Coxeter group. Theorem \ref{JustInfiniteCoxeter} tells us that $\Gamma$ is isomorphic to one of the graphs in Figure 1. Therefore by Proposition \ref {IrrAffineAb} $W_\Gamma\cong\mathbb{Z}^k\rtimes F$ and moreover by Lemma \ref{profiniteproducts} $\widehat{W_\Gamma}\cong\widehat{\mathbb{Z}}^k\rtimes F$. 

Assume that there exists a Coxeter group $W_\Omega$ such that $\widehat{W_\Gamma}\cong\widehat{W_\Omega}$. Since $W_\Omega\hookrightarrow\widehat{W_\Omega}$ and $\widehat{W_\Omega}\cong\widehat{W_\Gamma} $ is virtually abelian, we know that $W_\Omega$ is also virtually abelian. By the characterization of virtually abelian Coxeter groups, see Theorem \ref{FiniteVirtuallyAbelian}, there exist a finite Coxeter group $W_{\Omega_1}$ and $W_{\Gamma_1},\ldots, W_{\Gamma_n}$ where $\Gamma_i$ are from Figure 1 such that 
$$W_\Omega\cong W_{\Omega_1}\times W_{\Gamma_1}\times\ldots\times W_{\Gamma_n}$$
Applying Lemma \ref{profiniteproducts} we obtain
$$\widehat{W_\Gamma}\cong\widehat{W_\Omega}\cong W_{\Omega_1}\times \widehat{W_{\Gamma_1}}\times\ldots\times \widehat{W_{\Gamma_n}}$$
The first step in the proof is to show that $W_{\Omega_1}$ is in fact trivial. By Theorem \ref{JustInfiniteCoxeter} we know that $W_\Gamma$ does not have non-trivial finite normal subgroups, hence by Lemma \ref{NoFiniteNormal} we know that $\widehat{W_\Gamma}$ and hence $\widehat{W_\Omega}$ don't either. So $W_{\Omega_1}$ has to be trivial.

The next step is to ensure that in the decomposition $\widehat{W_\Omega}\cong\widehat{W_{\Gamma_1}}\times\ldots\times \widehat{W_{\Gamma_n}}$, we actually have $n=1$, that is, there appears only 1 factor. To do so, we invoke Lemma \ref{abelianization}. Together with the classification of the affine irreducible Coxeter groups and their abelianizations, see Proposition \ref{IrrAffineAb}, this implies that the abelianization of $W_\Gamma$ contains a factor isomorphic to $(\Z/2\Z)^n$ since each factor $W_{\Gamma_i}$ contributes at least one factor $\Z/2\Z$. Since $W^{ab}_\Gamma\cong(\mathbb{Z}/2\mathbb{Z})^t$ where $t\leq 3$ we obtain $n\leq 3$.

Since we have a complete list of possibilities for $W_\Gamma$ we can conclude that $W_\Omega$ consists of only one factor unless possibly $W_\Gamma$ is of type $\widetilde{A}_1,\widetilde{B}_m, \widetilde{F}_4, \widetilde{G}_2$ or $\widetilde{C}_l$ for $m\geq 3$, $l\geq 2$. 

Assume for contradiction that $n\geq 2$. Then we have 
$W_\Omega\cong W_{\Gamma_1}\times W_{\Gamma_2}$ or
$W_\Omega\cong W_{\Gamma_1}\times W_{\Gamma_2}\times W_{\Gamma_3}$
where $W_{\Gamma_i}\cong \mathbb{Z}^{a_i}\rtimes W_{\Delta_i}$, $W_{\Delta_i}$ is a finite Coxeter group.

By Lemma \ref{product}
$$W_\Omega\cong (\Z^{a_1}\times \Z^{a_2})\rtimes (W_{\Delta_1}\times W_{\Delta_2}) $$
or
$$W_\Omega\cong (\Z^{a_1}\times \Z^{a_2}\times \Z^{a_3})\rtimes(W_{\Delta_1}\times W_{\Delta_2}\times W_{\Delta_3})$$

Moreover by Lemma \ref{factoriso} we know that $F\cong W_{\Delta_1}\times W_{\Delta_2}$  or $F\cong W_{\Delta_1}\times W_{\Delta_2}\times W_{\Delta_3}$. In particular $F$ is directly decomposable into a product of non-trivial finite Coxeter groups. Now we can apply Theorem \ref{Nuida} to see that $F\cong W_{G_2}\cong\Z/2\Z\times W_{A_2}$ or $F\cong W_{B_{2k+1}}\cong \Z/2\Z\times W_{D_{2k+1}}$. Thus:
\begin{enumerate}
\item If $F\cong W_{G_2}$, then $W_\Gamma\cong\Z^2\rtimes W_{G_2}$ and $W_\Omega\cong(\Z\rtimes\Z/2\Z)\times(\Z^2\rtimes W_{A_2})$. Hence $\widehat{W_\Gamma}$ is virtually $\widehat{\Z}^2$ and $\widehat{W_\Omega}$ is virtually $\widehat{\Z}^3$ and these groups can not be isomorphic by Lemma \ref{VirAbelian}. Hence $n=1$.
\item If $F\cong W_{B_{2k+1}}$, then a similar argumentation as above shows that $n=1$.
\end{enumerate}
So far we have proved that $W_\Omega\cong\Z^{a_1}\rtimes W_{\Delta_1}$. 

Since $\widehat{W_\Gamma}\cong\widehat{W_\Omega}$, via Lemma \ref{abelianization}
we obtain $W^{ab}_\Gamma\cong W^{ab}_\Omega$, by Lemma \ref{VirAbelian} we get $|F|=|W_{\Delta_1}|$ and  $k=a_1$. Hence applying Corollary \ref{AffineCoxeterGroups} we conclude $W_\Gamma\cong W_\Omega$ which finishes the proof.
\end{proof}

\section{Profinite rigidity of the infinite Dihedral group}
Let $\Gamma$ be of type $\widetilde{A}_1$. Then the Coxeter group $W_\Gamma$ is known as the infinite Dihedral group. The next result is known in the community. Since we could not find a good reference we add the proof of it here.
\begin{proposition}
The infinite Dihedral group $D_\infty$ is profinitely rigid in the absolute sense.
\end{proposition}
\begin{proof}
The profinite completion of $D_\infty$ is isomorphic to $\widehat{\Z}\rtimes\Z/2\Z$ by Lemma \ref{profiniteproducts}.
Let $G$ be a finitely generated residually finite group with $\widehat{G}\cong \widehat{\Z}\rtimes\Z/2\Z$.
Since $G\hookrightarrow \widehat{\Z}\rtimes\Z/2\Z$, it follows that $G$ is a virtually  abelian group. By assumption $G$ is finitely generated, thus $G$ is virtually free abelian. Hence there exists $N\trianglelefteq G$ such that $N\cong\Z^n$ and $N$ is of finite index in $G$. By Lemma \ref{VirAbelian} it follows that $n=1$ and therefore $G$ is virtually infinite cyclic. It was shown in \cite[Lemma 2.5]{Farell} that a virtually infinite cyclic group $H$ has the form $F\rtimes\Z$ where $F$ is finite, or $H$ maps onto $D_\infty$ with a finite kernel. 

Assume first that $G$ maps onto $D_\infty$, then we apply \cite[Lemma 4.1]{BridsonConder} to conclude that $G\cong D_\infty$.
And if $G\cong F\rtimes\Z$ where $F$ is non-trivial finite, then $\widehat{G}\cong F\rtimes\widehat{\Z}$. In particular $\widehat{G}$ has a finite non-trivial normal subgroup $F$, but $\widehat{D_\infty}$ does not have any non-trivial finite normal subgroups by Lemma \ref{NoFiniteNormal}. The last case to consider is where $G\cong\mathbb{Z}$. In this case $\widehat{G}$ is abelian and therefore $\widehat{G}$ can not be isomorphic to $\widehat{D_\infty}$.
\end{proof}

\end{document}